\documentclass[12pt]{amsart}

\usepackage[english]{babel}
\usepackage{tikz-cd}
\usepackage{amsrefs}
\usepackage{graphicx}
\graphicspath{ {images/} }
\usepackage[T1]{fontenc}
\usepackage{txfonts}
\usepackage{times}
\usepackage{amssymb,verbatim,amscd,amsfonts,amsmath,amsthm,enumerate}
\usepackage[all]{xy}
\usepackage{subfig}
\usepackage{tikz}
\usetikzlibrary{shapes,arrows,shadows}
\usetikzlibrary{decorations.markings}
\usepackage{color}
\usepackage[normalem]{ulem}
\usepackage{hyperref}
\usepackage{wrapfig}
\usetikzlibrary{arrows}
\usepackage{pb-diagram}

\newtheorem{theorem}{Theorem}[section]

\newtheorem{corollary}[theorem]{Corollary}

\newtheorem{remark}[theorem]{Remark}

\newtheorem{definition}[theorem]{Definition}
\newtheorem{question}[theorem]{Question}

\begin{document}

\title{On Veech groups of infinite superelliptic curves}
\author[Camilo Ram\'{\i}rez Maluendas]{Camilo Ram\'{\i}rez Maluendas}
\address{Departamento de Matem\'atica y Estad\'istica\\ Facultad de Ciencias Exactas y Naturales\\ Universidad Nacional de Colombia, Sede Manizales \\
Manizales 170004, Colombia.}
\email{camramirezma@unal.edu.co}

\keywords{Infinite superelliptic curve, Loch Ness monster, flute surface, Veech group, tame translation surface, saddle connection, holonomy vectors} 
\subjclass[2000]{20F65; 37F10; 20H10; 57N16}

\begin{abstract}
We study infinite superelliptic curves as translation surfaces and explore their Veech groups. These objects are branched covering of the complex plane with branching over infinitely many points. We provide a criterion for isomorphism between a special family of infinite superelliptic curves. We show geometric descriptions of saddle connections and holonomy vectors on these infinite superelliptic curves. We prove that the Veech group of an infinite superelliptic curve are all the matrices arising from the differential of the affine mappings $\mathbb{C}$ to itself, permuting the branched points. We obtain necessary and sufficient conditions to guarantee that the Veech group of an infinite superelliptic curve is uncountable. We establish a trichotomy on the holonomy vector set and from it, we give a precise characterization of some countable groups that can appear as Veech group of an infinite superelliptic curve.  We also construct and study several examples of interesting infinite superelliptic curves illustrating our results.
\end{abstract}

\maketitle

\section*{Introduction}\label{sec:introduction}

The work of several authors in the field of \emph{hyperlliptic curves} has contributed to the advance of the theory of compact Riemann surfaces as complex algebraic curves \cites{Harvey, Rose}. Recently, in \cite{AGHQR}, this concept was extended within the context of the Weierstrass's map, given rise to a mathematical object known as \emph{infinite superelliptic curves}. More precisely, for any sequence $(z_{k})_{k\in\mathbb{N}}$ of different complex numbers  such that $\lim_{k\to\infty}\vert z_{k}\vert =\infty$, a Weierstrass's theorem ensures the existence of an entire map $f:\mathbb{C}\to \widehat{\mathbb{C}}$ (called \emph{Weierstrass map}) whose simple zeros are the points $z_{1},z_{2},\ldots$. Then, the affine plane curve
\begin{equation}\label{eq:infinite_superelliptic_curve_introduction}
S(f)=\left\{(z,w)\in\mathbb{C}^2: w^n=f(z)\right\},
\end{equation}
with $n\geq 2$, is called \emph{infinite superelliptic curve}. If $n=2$, then the affine plane curve is called \emph{infinite hyperelliptic curve}. This infinite superelliptic curve is a connected Riemann surface topologically equivalent to the \emph{Loch Ness monster}.  Moreover, the projection map onto the first coordinate $p_{\textbf{z}}:S(f)\to \mathbb{C}$, which is given by 
\begin{equation}\label{eq:projection_map_01}
 (z, w)\mapsto z,
\end{equation}
it is a branched covering map, whose branch points are the terms in the sequence $(z_{k})_{k\in\mathbb{N}}$. In addition, for each point $z\in \mathbb{C}\setminus \{ z_{k}: k\in \mathbb{N}\}$, the fiber $p_{\textbf{z}}^{-1}(z)$ consists of $n$ points. From these properties we obtain a \emph{tame translation structure} on $S(f)$ as we stated in Theorem \ref{theorem:tame_translate_strucuture_on_S(f)}.Furthermore, translation surfaces have been studied in different contexts: Dynamical system \cites{NPS, HopWe}, Riemann surface \cites{AbMu, MirWr} and Teichm\"uller theory \cites{Moller1, Shin} among others. These diverse research motivate the study of infinite superelliptic curve as translation surfaces. In the following, we will detail the results we have achieved.


If we consider the sequences $(z_{n})_{n\in\mathbb{N}}$ and $(z_{k}')_{k\in\mathbb{N}}$ such that $\lim_{k\to\infty}\vert z_{k}\vert =\infty$, and $\lim_{k\to\infty}\vert z_{k}'\vert =\infty$, then we take the   Weierstrass maps $f$ and $g$ having as simple zeros the points $z_{1},z_{2},\ldots$ and $z_{1}',z_{2}',\ldots$, respectively. Thus, we obtain the infinite hyperelliptic curves $S(f)$ and $S(g)$ as in equation \eqref{eq:infinite_superelliptic_curve_introduction}. In Theorem \ref{theorem:isomorphism}, we describe a criterion for isomorphism between infinite hyperlliptic curves, consisting by fiber-preserving to descend. More precisely, there exists a fiber-preserving to descend isomorphism $\tilde{T}: S(f)\to S(g)$, if and only if, there exists an isomorphism  $T:\mathbb{C}\to \mathbb{C}$ sending the zeros of $f$ onto the zeros of $g$. 

The Banach space of all the bounded sequence of complex numbers is denoted by $\ell_{\infty}$, then the collection of infinite hyperelliptic curve can be thought as $\mathcal{C}_{0}$ the subset of $\ell_{\infty}$ of all sequence of different complex numbers such that each one of them converge to zero. Notably,  the group of complex translation mappings acts on $\mathcal{C}_{0}$. Consequently, as an outcome of the preceding result, the space of all infinite hyperelliptic curves having tame translation surfaces, up to isomorphism, it can be identified to the quotient space $\mathcal{C}_{0}/ {\rm Trans}(\mathbb{C})$ (refer to Theorem \ref{t:modulli_space}).

It is well-known that the set of holonomy vectors, denoted as ${\rm Hol}(S)$, of a compact translation surface $S$, it does not have limit points \cite{Voro}*{Proposition 3.1}. Nevertheless, in the case of an infinite superelliptic curve $S(f)$, we establish the following \emph{trichotomy} on the set ${\rm Hol}(S(f))$:
\begin{itemize}
    \item[\textbf{(1)}] The set of holonomy vectors ${\rm Hol}(S(f))$ of  $S(f)$ does not have limit points, it means, the derived set ${\rm Hol}(S(f))'=\emptyset$. Therefore, ${\rm Hol}(S(f))$ is discrete.
    \item[\textbf{(2)}] The derived set ${\rm Hol}(S(f))'\neq \emptyset$ and ${\rm Hol}(S(f))$ is a closed subset of $\mathbb{C}$.
    \item[\textbf{(3)}] The derived set ${\rm Hol}(S(f))'\neq \emptyset$ and ${\rm Hol}(S(f))$ is not a closed subset of $\mathbb{C}$.
\end{itemize}
For each of the aforementioned cases, we present an example of an infinite superelliptic curve (see \S\ref{example:holonomy_Vectors_origami},  \S\ref{example:isc_holonomy_omega}, \S\ref{example:holonomy_set_rationals}). Nonetheless, the precise characterization of the set of holonomy vector associated to the infinite superelliptic curves in these examples is achieved through  Theorem \ref{theorem:holonomy_vectors_of_S(f)}. This theorem establishes that the holonomy vectors are certain differences between the zeros of the Weierstrass function $f$. In other words, the element $v\in\mathbb{C}$ is a holonomy vector of the infinite superelliptic curve $S(f)$, if and only if, there are zeros $z_{r}$ and $z_{l}$ of the Weierstrass function $f$ such that $v=z_{r}-z_{l}$ and, the straight line segment joining the points $z_{r}$ and $z_{l}$ contains no zeros of $f$ within it. We finish our study of an infinite superelliptic curve $S(f)$ as translation surface, by proving a geometric description of all its saddles connection lifting suitable straight line segments of $\mathbb{C}$ via the projection map $p_{\textbf{z}}:S(f)\to\mathbb{C}$. In Theorem \ref{theorem:saddle_connection_S(f)} we establish that a geodesic interval $\gamma$ of $S(f)$  is a saddle connected, if and only if, the image $p_{\textbf{z}}(\gamma):=L$ is a straight line segment joining two zeros of the Weierstrass $f$ such that $L$ removing its ends point does not pass throughout any zero of $f$. 

In 1989 W. A. Veech \cite{Vee} associated to each translation surface $S$ a group of matrices $\Gamma(S)<{\rm GL}_{+}(2,\mathbb{R})$, which is well known as the \emph{Veech group of $S$}. This group has attracted the attention of several mathematicians \cites{EskMirz, Celta, McMuller, Bain}. In particular, it has been studied the Veech group of a class of special translation surfaces that are covering of the square torus ramified over the origin, so-called \emph{square tiled surfaces} or \emph{origamis}  \cites{GutJud, HopWe2012, GabiT}. In this paper, we explore the Veech group $\Gamma(S(f))$ of an infinite superelliptic curve $S(f)$ through branched covering map (projection map) $p_{\textbf{z}}$. We take the (translation) flute surface $\mathbb{C}^{\star}$ obtained from the complex plane $\mathbb{C}$ by removing the zeros of the Weierstrass map $f$. Then our Theorems \ref{theorem:related_Vecch_group_1} and \ref{theorem:related_Vecch_group_2} ensure that the Veech groups associated to $S(f)$ and $\mathbb{C}^{\star}$, respectively, are the same. In other words,  $\Gamma(S(f))=\Gamma(\mathbb{C}^{\star})$. 

In \cite{PSV}*{Theorem}, the authors proved that if the Veech group $\Gamma(S)$ of a tame translation surfaces $S$ is a proper subgroup of ${\rm GL}_{+}(2,\mathbb{R})$, then one of the following holds:
\begin{itemize}
    \item[\textbf{(i)}] The group $\Gamma(S)$ is uncountable and conjugate to a special group of upper triangular matrices, or 
    \item[\textbf{(ii)}] The group $\Gamma(S)$ is countable and without contracting elements.
\end{itemize}
Naturally, this leads us inquire: which groups can appear as Veech groups within the context of infinite superelliptic curve? For the case \textbf{(i)}, Theorem \ref{t:characterization_uncountable} establishes sufficient and
necessary conditions to guarantee that the Veech of an infinite superelliptic curve $S (f)$ is an uncountable subgroup of ${\rm GL}_{+}(2, \mathbb{R})$. To be more specific, the Veech group $\Gamma(S(f))$ of $S(f)$ is an uncountable subgroup of ${\rm GL}_{+}(2,\mathbb{R})$, if and only if, all the zeros of the Weierstrass map $f$ are collinear. Furthermore, in Subsection \ref{sub:realization_uncountable_Veech_group}, we realize certain uncountable subgroups of ${\rm GL}_{+}(2,\mathbb{R})$ as Veech group of infinite superelliptic curves.

The case \textbf{(ii)} is studied from trichotomy established on the set of holonomy vector.  If the Veech group $\Gamma(S(f))$ of the infinite superelliptic curve $S(f)$ is an non-trivial countable subgroup  of ${\rm GL}_{+}(2,\mathbb{R})$ without contracting matrices. Then Theorem \ref{theorem:countable_Veech_group_of_infinite_superelliptic_curve} establishes that:

\begin{itemize}
\item[\textbf{(1)}] If the set of holonomy vectors ${\rm Hol}(S(f))$ does not have limits point, then the Veech group $\Gamma(S(f))$ is discrete.

\item[\textbf{(2)}] If the set of holonomy vectors ${\rm Hol}(S(f))$ is closed such that the derived set ${\rm Hol}(S(f))'\neq \emptyset$, then the Veech group $\Gamma(S(f))$ is homeomorphic to some finite subgroup of $S^{1}$ the unit circle in the complex plane, considered with complex multiplication.
\end{itemize}
In Subsection \ref{example:finite_group_as_Veech_group}, we realize any finite group of ${\rm GL}_{+}(2,\mathbb{R})$ as Veech group of an infinite superelliptic curve.

The paper is structured as follows. In Section \ref{sec:tame}, we provide a introduction to several concepts: tame translation surface, infinite superelliptic curve, and present the equivalent relation isomorphic on infinite hyperelliptic curves. Moreover, we recall the objects: saddles connection and holonomy vectors of infinite. We also state and proof our main results related to the properties of infinite superelliptic curves as translation surfaces. In Section \ref{section:Veech_group} we collect the concept of Veech group associated to translation surfaces. Here, we focus on the Veech group of infinite superelliptic curves.

\subsection*{Acknowledgements}

We express our gratitude to Universidad Nacional de Colombia, Sede Manizales. We have dedicated this work to our beautiful family: Marbella and Emilio, in appreciation of their love and support. We thank Israel Morales Jim\'enez and John Alexander Arredondo for providing helpful comments and conversations
on this work.


\section{Infinite superelliptic curves as tame translation surfaces}\label{sec:tame}

We shall identify throughout  the whole manuscript, the complex plane $\mathbb{C}$ with the real $2$-dimensional space $\mathbb{R}^2$ by sending $\{1,i\}$  to the standard  basis of $\mathbb{R}^2$. The term \emph{surface} refers to a connected 2-dimensional topological real manifold having empty boundary, and it is denoted by $S$.

\subsection{Tame translation surfaces}\label{subsection_tame_translation_surface}
A surface $S$ is said to be \emph{translation}, if $S$, except for a subset of points ${\rm Sing}(S)\subset S$, can be endowed with a translation atlas, \emph{i.e.}, an atlas whose transition maps are locally restrictions of a translation of $\mathbb{C}$. We assume that every point in ${\rm Sing}(S)$ is non removable, it means, the translation atlas can not be extended to any of the points in ${\rm Sing}(S)$. An element $x$ of the set ${\rm Sing}(S)$ is called \emph{singular point of $S$} or \emph{singularity}. A \emph{translation structure} on $S$ is a maximal translation atlas on $S$. From the Uniformization Theorem \cite{Abi}*{p. 580}, it follows that the only surfaces that admit a translation structure without removing points are: the plane, the torus and the cylinder \cite{FarKra}*{p. 193}. 

If we consider a translation surface $S$, then we can pull back the Euclidean (Riemannian) metric of $\mathbb{C}$ via its translation structure; thus we obtain a flat Riemannian metric $\mu$ on $S$. Let $\widehat{S}$ denote the \emph{metric completion of $S$} with respect to the flat Riemannian metric $\mu$.

\begin{definition}[\cite{PSV}*{Definition 2.2}]\label{definition:tame}
A translation surface $S$ is said to be \textbf{tame}, if for each point $x\in \widehat{S}$ there exists a neighborhood $U_x \subset \widehat{S}$ isometric to either:
\begin{itemize}
	\item[\textbf{(1)}] Some open of the complex plane.
	
	\item[\textbf{(2)}] An open of the ramification point of a (finite or infinite) cyclic branched covering of the unit disk in the complex plane.
\end{itemize}
\end{definition}

In the later case, if the neighborhood $U_{x}$ is isometric to the finite cyclic branched covering, which has finite order $m\in\mathbb{N}$, then the point $x$ is called \emph{finite cone angle singularity of angle $2m\pi$}. But, if $U_{x}$ is isometric to the infinite cyclic branched covering, then the point $x$ is called \emph{infinite cone angle singularity}.

We denote by ${\rm Sing}(\widehat{S})$ the set of all the finite and infinite cone angle singularities of $\widehat{S}$. Each element of ${\rm Sing}(\widehat{S})$ is called \emph{a cone angle singularity of} $\widehat{S}$ or just \emph{cone point}. We note that the space  obtained from $\widehat{S}$ by removing its infinite cone angle singularities, is topologically equivalent to the surface $S$. 

\subsection{Infinite superelliptic curves}\label{subsection:infinite_hyperelliptic_curve}
A sequence  $(z_{k})_{k\in\mathbb{N}}$ of different complex numbers is said to be \emph{convergent to infinity}, if the sequence of its norm $(\vert z_{k}\vert)_{k\in\mathbb{N}}$ converges to infinity. In other words, it satisfies $\lim\limits_{k\to \infty}\vert z_{k}\vert=\infty$. Throughout this paper, we will assume that, unless otherwise noted, the sequence $(z_{k})_{k\in\mathbb{N}}$ is ordered as follows: $ \vert z_k\vert < \vert z_{k+1}\vert$ and, in the case  $\vert z_k\vert =\vert z_{k+1}\vert$, then $0\leq \arg(z_k)<\arg(z_{k+1})<2\pi$. 

We denote by $\widehat{\mathbb{C}}$ the extended complex plane. If we consider  a convergent sequence  to infinity $(z_{k})_{k\in\mathbb{N}}$, then the Weierstrass Theorem in \cite{Palka}*{p. 498} guarantees the existence of a meromorphic function $f:\mathbb{C}\to \widehat{\mathbb{C}}$ (called \emph{Weierstrass map}) whose simple zeroes are given by the points $z_{1},z_{2},\ldots$. Moreover, the map $f$ is uniquely determined (up multiplication) by a zero-free entire map, and it can be written as
\begin{equation}\label{eq:T_Weierstrass}
f(z)=h(z)z^{m}\prod_{k=1, z_{k} \neq 0}^{\infty}\left(1-\frac{z}{z_k}\right)E_k(z),
\end{equation}
where $h$ is a zero-free entire function ($m=0$, if $z_{k} \neq 0$ for every $k\in\mathbb{N}$; in the other case, $m=1$ for $z_{k}=0$), 
and $E_{k}(z)$ is a function of the form
\[
E_{k}(z)=\exp \left[\sum_{s=1}^{d(k)}\frac{1}{s}\left(\frac{z}{z_k}\right)^s \right],
\]
for a suitably large non-negative integer $d(k)$. 

Now, if we consider the holomorphic function $F:\mathbb{C}^2\to \mathbb{C}$ given by 
\[
F(z,w)=w^{n}-f(z),
\]
with $n\in\mathbb{N}$ such that $n\geq 2$, then we obtain the \emph{affine plane curve} 

\begin{equation}\label{equation-definition:Infinite_hyperelliptic_curve}
	S(f)=\left\{(z,w)\in\mathbb{C}^2: w^n=f(z)\right\}.
	\end{equation}
 
\begin{definition}[\cite{AGHQR}*{Subsection 6.3.1}]\label{definiton:infinite_hyperelliptic_curve}
Then the affine curve $S(f)$ is a Riemann surface, called {\bf infinite superelliptic curve}. If moreover, $n=2$, the affine plane curve $S(f)$ is known as {\bf infinite hyperelliptic curve}.
\end{definition}	

The following results describes the topology type of an infinite superelliptic curve.

\begin{theorem}[\cite{AGHQR}*{Theorem 6.2}]
	 The infinite superelliptic curve $S(f)$ is a connected Riemann surface homeomorphic to the Loch Ness monster.
\end{theorem}

Recall that  the \emph{Loch Ness monster} is the unique, up to homeomorphism, infinite genus surface with exactly one end \cite{PSul}. Now, we shall define a tame translation structure on the infinite superelliptic curve $S(f)$. It is necessary to introduce the following remark.

 \begin{remark}[\cite{AGHQR}*{Remark 6.11}]\label{remark:properties_projection_map_first_coordinate}
 The projection map onto the first coordinate $p_{\textbf{z}}:S(f)\to \mathbb{C}$, which is given by $
 (z, w)\mapsto z$, satisfies the following properties:

\begin{itemize}
\item[\textbf{(1)}] It is a branched covering map, whose branch points are the terms in the sequence $(z_{k})_{k\in\mathbb{N}}$. 

\item[\textbf{(2)}] For each point $z\in \mathbb{C}\setminus \{ z_{k}: k\in \mathbb{N}\}$, the fiber $p_{\textbf{z}}^{-1}(z)$ consists of $n$ points.
\end{itemize}
\end{remark}

\begin{theorem}\label{theorem:tame_translate_strucuture_on_S(f)}
The infinite superelliptic curve $S(f)$ admits a tame translate structure. Moreover, each ordered pair $(z_{k},\textbf{0})$ of $ S(f)$ is a finite cone angle singularity of angle $2m\pi$, with $k\in\mathbb{N}$.
\end{theorem}

\begin{proof}
We remove the points of the sequences $(z_{k})_{k\in\mathbb{N}}$ and the points $(z_{k},\textbf{0})$, for each $k\in\mathbb{N}$, respectively, from the complex plane $\mathbb{C}$ and the infinite superelliptic curve $S(f)$, respectively. Thus, we obtain the subsurfaces 
\begin{align}
\mathbb{C}^{\star} &:=\mathbb{C}\setminus \{z_{k}:k\in\mathbb{N}\},\label{eq:flute_surface_C}\\
S(f)^{\star}&:=S(f)\setminus\{(z_{k},\textbf{0}):k\in\mathbb{N}\}.\label{eq:flut_sruface_S(f)}
\end{align}
As the complex plane $\mathbb{C}$ is considered as a translation surfaces with the global chart ${\rm Id}$. Then, the open subset $\mathbb{C}^{\star}\subset \mathbb{C}$ inherits this translation structure.

On the other hand, from the Remark \ref{remark:properties_projection_map_first_coordinate}, it follows that the restriction of the projection map onto the first coordinate 
\begin{equation}\label{eq:unramiried_covering_map_pi_z}
p_{\textbf{z}}:S(f)^{\star}\to \mathbb{C}^{\star},
\end{equation}
is an unramified covering map of finite degree $n$. Using this covering function $p_{\textbf{z}}$, we lift the translation structure of $\mathbb{C}^{\star}$ on $S(f)^{\star}$. Now, we denote by $\mu$ the flat Riemann metric on $S(f)^{\star}$, which is defined by its translation structure. Then, Remark \ref{remark:properties_projection_map_first_coordinate} guarantees that the metric completion of $S(f)^{\star}$ with respect to $\mu$ is $S(f)$, where each ordered pair $(z_{k},\textbf{0})$ of $S(f)$ is a finite cone point of angle $2m\pi$, for each $k\in\mathbb{N}$.
\end{proof}
  
\begin{remark}\label{remark:properties_of-C_star_S(f)_star}
The translation surfaces $\mathbb{C}^{\star}$ and $S(f)^{\star}$ defined in equations \eqref{eq:flute_surface_C} and \eqref{eq:flut_sruface_S(f)}, respectively, have the following topology type:
\begin{itemize}
\item[\textbf{(1)}] The surface $\mathbb{C}^{\star}$ is (up homeomorphism) a flute surface.

\item[\textbf{(2)}] The surface $S(f)^{\star}$ has as ends space ${\rm Ends}(S(f)^{\star})$ the ordinal number $\omega+1$, and the corresponding end to the accumulate point is the only one among them that is a non-planar end. 
\end{itemize}
\end{remark}

\subsection{Isomorphism}  Let $S_{1}$ and $S_{2}$ be tame translation surfaces with translation atlas $\mathcal{A}=\{(U_{i}, \phi_{i})\}_{i\in\mathcal{I}}$ and $\mathcal{B}=\{(V_{j},\psi_{j})\}_{j\in\mathcal{J}}$, respectively. A homeomorphism $T:S_{1}\to S_{2}$ is called \emph{isomorphism from the structure $\mathcal{A}$ to the structure $\mathcal{B}$}, if it satisfies the following two properties:
\begin{itemize}
    \item[\textbf{(1)}] The map $T$ sends the singular points of $S_1$ onto the singular points of $S_2$, \emph{i.e.}, $T({\rm Sing}(S_1))={\rm Sing}(S_2)$.
    \item[\textbf{(2)}] The translation structure $\mathcal{B}=\{(T(U_{i}),\phi_{i} \circ T^{-1}: T(U_{i})\to \mathbb{C})\}_{i\in\mathcal{I}}$.
\end{itemize}
Two tame translation surfaces $S_{1}$ and $S_{2}$ are \emph{isomorphic}, if there is an isomorphism $T:S_{1}\to S_{2}$. 

Now, we will state a criterion for isomorphism between infinite hyperlliptic curves, consisting by fiber-preserving to descend. 

Let $(z_{k})_{k\in\mathbb{N}}$ and $(z_{k}')_{k\in\mathbb{N}}$ be two convergent sequence to infinity. Let $f$ and $g$ be the  Weierstrass maps as in the equation \eqref{eq:T_Weierstrass} having as simple zeros the points $z_{1},z_{2},\ldots$ and $z_{1}',z_{2}',\ldots$, respectively. Then we obtain the following infinite hyperelliptic curves
\begin{align}
    S(f)&=\left\{(z,w)\in\mathbb{C}^2:w^2=f(z)\right\},\\
    S(g)&=\left\{(z,w)\in\mathbb{C}^2:w^2=g(z)\right\}.
\end{align}

On the other hand, let us remember that: an isomorphism $T:S(f)\to S(g)$ is called \emph{fiber-preserving to descend}: if 
 the ordered pairs $(u_{1},v_{1})$ and $(u_{2},v_{2})$ are in the same fiber $p_{\textbf{z}}^{-1}(z)$, for any $z\in \mathbb{C}\setminus\{z_{k}:k\in\mathbb{N}\}$, then it satisfies the equality $p_{\textbf{z}}\circ T(u_{1},v_{1})=p_{\textbf{z}}\circ T(u_{2},v_{2})$, where $p_{\textbf{z}}$ is the projection map defined in Remark \ref{remark:properties_projection_map_first_coordinate}.

A necessary and sufficient condition for the infinite hyperelliptic curves $S(f)$ and $S(g)$ to be isomorphic is given by:

\begin{theorem}\label{theorem:isomorphism}
There exists a fiber-preserving to descend isomorphism $\tilde{T}: S(f)\to S(g)$, if and only if, there exists an isomorphism  $T:\mathbb{C}\to \mathbb{C}$ sending the zeros of $f$ onto the zeros of $g$.
\end{theorem}

\begin{proof}
We shall prove the sufficient condition. We assume that there exists a fiber-preserving descend isomorphism  $\tilde{T}:S(f)\to S(g)$. Since the translation structure on $S(f)$ is defined by the projection map $p_{\textbf{z}}$ (see Theorem  \ref{theorem:tame_translate_strucuture_on_S(f)}), then it is immediate that the map $T:=p_{\textbf{z}}\circ \tilde{T}\circ p_{\textbf{z}}^{-1}$ is a well defined isomorphism from $\mathbb{C}$ onto itself, which sends the points of the sequence $(z_{k})_{k\in\mathbb{N}}$ onto the points of the sequence $(z'_{k})_{k\in\mathbb{N}}$. This proves the sufficient condition.

We shall show the necessary condition. By hypothesis, there is an isomorphism $T:\mathbb{C}\to \mathbb{C}$ sending the points of the sequence $(z_{k})_{k\in\mathbb{N}}$ onto the the points of the sequence $(z'_{k})_{k\in\mathbb{N}}$. Then the map  $\tilde{T}: S(f)\to S(g)$ will be defined as a lifting of the map $T$. For each $k\in\mathbb{N}$, the image of the cone angle singularity $(z_{k},\textbf{0})$ under $\tilde{T}$ is given by 
 \[
 \tilde{T}(z_{k},\textbf{0}):=(T(z_{k}),\textbf{0}).
 \] 
 Now, we choose a basepoint $(z_{0},w_{0})$ in the subsurface $ S(f)^{\star}$, where $S(f)^{\star}:=S(f)\setminus {\rm Sing}(S(f))$, and fix an ordered pair $(u,v)$ in the fiber $p_{\textbf{z}}^{-1}(T(z_{0}))$. Thus, we define the image of $(z_{0},w_{0})$ under $\tilde{T}$ as
 \[
 \tilde{T}(z_{0}, w_{0})=(u,v).
 \]
 For each ordered pair  $(s,t)\in S(f)^{\star}$, we consider a path $\gamma$ in the flute surface $\mathbb{C}\setminus\{z_{k}:k\in\mathbb{N}\}$ joining the complex numbers $z_{0}$ and $s$. For the composition map $T\circ \gamma$ there is a lifting path $\tilde{\gamma}$ in $S(g)$ joining the ordered pairs $\tilde{T}(z_{0},w_{0})$ and $(\tilde{s},\tilde{t})$, for some  $(\tilde{s},\tilde{t})$ in the fiber $p_{\textbf{z}}^{-1}(T(s))$. Thus, the image of $(s,t)$ under $\tilde{T}$ is defined as
 \[
 \tilde{T}(s,t):=(\tilde{s},\tilde{t}).
 \]
 By construction the map $\tilde{T}:S(f)\to S(g)$ is a fiber-preserving to descend homeomorphism. Given that $T$ is an isomorphism of the complex plane $\mathbb{C}$ onto itself and the translation structures on $S(f)$ and $S(g)$, respectively, are defined from the projection map $p_{\textbf{z}}$, then the map $\tilde{T}$ is also an isomorphism.
 \end{proof}

We remark that the group of all the isomorphism from $\mathbb{C}$ to itself is the group of translations of the complex plane 
\[
{\rm Trans}(\mathbb{C})=\left\{ T_{b}:\mathbb{C}\to\mathbb{C}, \text{ given by } T_{b}(z)=z+b \mid  b\in\mathbb{C}\right\}.
\]
If we take $(z_{k})_{k\in\mathbb{N}}$ a convergent sequence to infinity, then the sequence $(T_{b}(z_{k}))_{k\in\mathbb{N}}=(z_{k}+b)_{k\in\mathbb{N}}$ is also convergent to infinity. Hence, we consider the Weierstrass maps $f$ and $f_{b}$ as in equation \eqref{eq:T_Weierstrass} having as simple zeros the points $z_{1},z_{2},\ldots$ and $z_{1}+b,z_{2}+b,\ldots$, respectively. Then, we obtain the infinite hyperelliptic curves
\begin{align}
    S(f)&=\left\{(z,w)\in\mathbb{C}^2:w^2=f(z)\right\},\\
    S(f_{b})&=\left\{(z,w)\in\mathbb{C}^2:w^2=f_{b}(z)\right\}.
\end{align}
Theorem \ref{theorem:isomorphism} ensures that the infinite hyperelliptic curves $S(f)$ and $S(f_{b})$ are isomorphic. 

On the other hand, let $\mathcal{C}_{\infty}$ denote the set of all convergent sequence to infinity. It can be thought as the set of all the infinite hyperelliptic curves having tame translation structure. Now, let $\ell_{\infty}$ denote the Banach space of all the bounded sequence of complex numbers $(z_k)_{k\in\mathbb{N}}$ equipped with the norm given by the formula
\[
\Vert (z_k)_{k\in\mathbb{N}}\Vert_{\infty}=\sup\limits_{k\in\mathbb{N}} \vert z_k\vert <\infty.
\]
Let $\mathcal{C}_0$ be the subspace of $\ell_{\infty}$ of all sequences of different complex numbers $(w_{k})_{k\in\mathbb{N}}$ such that each one of them converge to zero and, $w_{k}\neq \textbf{0}$, for each $k\in\mathbb{N}$. Using the involution map $w\mapsto \frac{1}{w}$, we identify each sequence $(w_{k})_{k\in\mathbb{N}}$ of $\mathcal{C}_{0}$ with the sequence $\left(\frac{1}{w_{k}}\right)_{k\in\mathbb{N}}$ of $\mathcal{C}_{\infty}$. We remark that the group of translations ${\rm Trans}(\mathbb{C})$ acts on the space $\mathcal{C}_{0}$ by
\begin{equation}\label{eq:space_of_convergent_infinity_sequence}
(T_{b}, (w_{k})_{k\in\mathbb{N}})\mapsto \left(\frac{w_{k}}{1+bw_{k}}\right)_{k\in \mathbb{N}}.
\end{equation}
As a consequence of action above and  Theorem \ref{theorem:isomorphism}, we obtain the following result describing the space of infinite hyperelliptic curves defined as tame translation surfaces, up isomorphism.

\begin{theorem}\label{t:modulli_space}
The space of all infinite hyperelliptic curves having tame translation surfaces, up to isomorphism, it can be identified to the quotient space $\mathcal{C}_{0}/ {\rm Trans}(\mathbb{C})$.
\end{theorem}

\subsection{Saddles connection and holonomy vectors} \label{section:saddle-holonomy}

A \emph{saddle connection} $\gamma$ of a tame translation surface $S$ is a geodesic interval joining two cone points and having non cone points in its interior. On the translation structure of $S$, we can find a chart $(U,\varphi)$ such that the open $U$ contains the saddle connection $\gamma$ without its endpoints, then the map $\varphi$ sends $\gamma$ on a straight line segment in $\mathbb{C}$. We orient such straight line segment in two possible ways  $[\theta], [-\theta] \in \mathbb{R}/2\pi \mathbb{Z}$, for any $\theta\in\mathbb{R}$. Thus, we obtain two oppositely pointing vectors $\{v,-v\}\subset\mathbb{C}$ associated to the saddle connection $\gamma$, which are in the direction $[\theta]$ and $[-\theta]$, respectively, and their norm are equal to the length of $\gamma$, with respect to the flat Riemannian metric $\mu$ on $S$. Each one of these vectors is called a \emph{holonomy vector} of $\gamma$. Two saddle connection of $S$ is called \emph{parallel}, if their respective holonomy vectors are parallel. The set of all the holonomy vectors of all saddle connections of the tame translation surface $S$ is denoted by ${\rm Hol}(S)$. We remark that the zero vector $\textbf{0}$ is not in ${\rm Hol}(S)$.

Let $(z_{k})_{k\in\mathbb{N}}$  converge to infinity and let $f$ be a Weierstrass map as in equation \eqref{eq:T_Weierstrass} having as simple zeros the points $z_{1},z_{2},\ldots$. Then we define the infinite superelliptic curve 
\begin{equation}
S(f)=\{(z,w)\in\mathbb{C}^2:f(z)=w^n\},
\end{equation}
with $n\geq 2$, and take the projection map $p_{\textbf{z}}:S(f)\to\mathbb{C}$ given by $(z,w)\mapsto z$ as in Remark \ref{remark:properties_projection_map_first_coordinate}. In addition, we take the subsurfaces (see Remark \ref{remark:properties_of-C_star_S(f)_star})
\begin{align}
\mathbb{C}^{\star} &:=\mathbb{C}\setminus \{z_{k}:k\in\mathbb{N}\},\\
S(f)^{\star}&:=S(f)\setminus\{(z_{k},\textbf{0}):k\in\mathbb{N}\}.
\end{align}
The following result gives sufficient and necessary conditions to draw saddles connection on $S(f)$.

\begin{theorem}\label{theorem:saddle_connection_S(f)}
The  geodesic interval $\gamma$ of $S(f)$  is a saddle connected, if and only if, the image $p_{\textbf{z}}(\gamma):=L$ is a straight line segment in the complex plane $\mathbb{C}$ satisfying the following properties:
\begin{itemize}
    \item[\textbf{(1)}] Its endpoints are two different complex numbers $z_{r}$ and $z_{l}$ of the sequence $(z_{k})_{k\in\mathbb{N}}$.
    \item[\textbf{(2)}] The straight line segment $L$ without its endpoints is contained in $\mathbb{C}^{\star}$.
\end{itemize}
\end{theorem}

\begin{proof}
We shall prove the sufficient condition. We consider a saddle connection $\gamma$ of $S(f)$. Since set of cone points ${\rm Sing}(S(f))$ are all the ordered pairs $(z_{k},\textbf{0})$, for each $k\in\mathbb{N}$, then the saddle connected $\gamma$ has endpoints $(z_{r},\textbf{0})$ and $(z_{l},\textbf{0})$, for some $r\neq l\in\mathbb{N}$. Now, given that the tame translation structure on $S(f)$ is defined by the projection map $p_{\textbf{z}}$, then the function $p_{\textbf{z}}$ must send the saddle connection $\gamma$ to a straight line segment $L$ in the complex plane $\mathbb{C}$ satisfying the following properties:
   \begin{itemize}
    \item[\textbf{(1)}] Its endpoints are two different complex numbers $z_{r}$ and $z_{l}$ of the sequence $(z_{k})_{k\in\mathbb{N}}$.
    \item[\textbf{(2)}] The straight line segment $L$ without its endpoints is contained in the flute surface $\mathbb{C}^{\star}$, because the saddle connection $\gamma$ does not have cone points in its interior.
    \end{itemize} 
  
\begin{remark}
As the restriction map $p_{\textbf{z}}:S(f)^{\star}\to\mathbb{C}^{\star}$ is an unbranched covering map, then the inverse image $p_{\textbf{z}}^{-1}(L\setminus\{z_{r},z_{l}\})$ is conformed by $n$ disjoint curves $\gamma_{1},\ldots,\gamma_{n}$ of $S(f)^{\star}$. One of these curves must be the saddle connection $\gamma$ without its endpoints. 
\end{remark}

 The necessary condition is immediate, because the translation structure on $S(f)$ is defined via the projection map $p_{\textbf{z}}:S(f)\to\mathbb{C}$. 
\end{proof}

\begin{definition}\label{definitio:saddle_connection_sequence} 
    A straight line segment $L$ in the complex plane $\mathbb{C}$ is to be said \textbf{saddle connection of  the flute surface $\mathbb{C}^{\star}$}, if it satisfies items \textbf{(1)} and \textbf{(2)} described in the previously Theorem \ref{theorem:saddle_connection_S(f)}.
\end{definition}

We remark that the set of all saddle connections of the sequence $(z_{k})_{k\in\mathbb{N}
}$ is non-empty, because if we consider any two different points $z_{r}$ and $z_{l}$ of the sequence $(z_{k})_{k\in\mathbb{N}}$, then we draw the straight line segment $L$ with endpoints $z_{r}$ and $z_{l}$. As the points of the sequence $(z_{k})_{k\in\mathbb{N}} $ are discrete, then $L$ contains at most a finitely many points of the sequence $(z_{k})_{k\in\mathbb{N}}$. These points are ordered by the norm. Then, it is enough to take two of these points being consecutive, and thus, we obtain a saddle connection of the sequence $(z_{k})_{k\in\mathbb{N}}$.  

Using the previous fact and Theorem \ref{theorem:saddle_connection_S(f)}, we can guarantee the existence of saddle connections and holonomy vector of the infinite superelliptic curve $S(f)$. Even more, we can describe all the holonomy vectors of the infinite superelliptic curve $S(f)$ by certain difference of the points in the sequence $(z_{k})_{k\in\mathbb{N}}$.

\begin{theorem}\label{theorem:holonomy_vectors_of_S(f)}
Let $v$ be a vector of $\mathbb{C}$. Then, $v$ is a holonomy vector of the infinite superelliptic curve $S(f)$, if and only if, there is a saddle connection $L$ of the flute surface $\mathbb{C}^{\star}$ having endpoints $z_{r}$ and $z_{l}$, such that either 
\begin{equation}\label{eq:holonomy_vectors_formula}
v=z_{r}-z_{l} \quad \text{or} \quad -v=z_{l}-z_{r}.
\end{equation}
Moreover, the set of the holonomy vectors ${\rm Hol}(S(f))$ of the infinite superelliptic curve $S(f)$ is countable. 
\end{theorem}

\begin{proof}
We shall prove the sufficient condition. We fix a holonomy vector $v$, then we consider a saddle connection $\gamma$ of $S(f)$, having $v$ associated as holonomy vector.  The saddle connection $\gamma$ connects the cone points $(z_{r},\textbf{0})$ and $(z_{l},\textbf{0})$, for any $r<l\in\mathbb{N}$, then projection map $p_{\textbf{z}}$ sends $\gamma$ to the  saddle connection $L$ of the sequence $(z_{k})_{k\in\mathbb{N}}$ having endpoints $z_{r}$ and $z_{l}$. Given that the tame translation structure on $S(f)$ is determined by the projection map $p_{\textbf{z}}$, then the holonomy vectors associated to $\gamma$ are given by the formula
\begin{equation}\label{eq2:holonomy_vectors_formula}
u=z_{r}-z_{l} \quad \text{and } \quad -u=z_{l}-z_{r}.
\end{equation}
This implies either $v=z_{r}-z_{l}$ or $v=z_{l}-z_{r}$.

Conversely, given that every saddle connection $L$ of the flute surface $\mathbb{C}^{\star}$ defines $n$ saddle connections in $S(f)$, via the projection map $p_{\textbf{z}}$, then the holonomy vectors associated to these $n$ saddle connections are the vectors given by the differences of the endpoints of $L$. Thus, we have obtained the following fact.

Now, we shall prove that the set of holonomy vectors ${\rm Hol}(S(f))$ is countable. Given that ${\rm Hol}(S(f))$ is a subset of all difference $H:=\{z_{r}-z_{l}: \text{ for each } r\neq l\in\mathbb{N}\}$, then it is enough to prove that $H$ is countable. We consider $P=\{p_{k}\}_{k\in\mathbb{N}}$ the subset of all prime numbers. It is easy to check that the map $\psi: H\to \mathbb{Z}$ given by the following product
     \[
\psi(z_{r}-z_{l})= \left\{
\begin{array}{ ll }
 p_{r}\cdot p_{l}; & \text{ if } r<l; \\
-p_{r}\cdot p_{l}; & \text{ if } r>l.\\
\end{array}
\right.
\]
it is a well-defined and injective.
\end{proof}

As a consequence of the previous Theorem, we have the following result.

\begin{corollary}\label{corollary:vectors_collinear}
   The points of the sequence $(z_{k})_{k\in\mathbb{N}}$ are collinear, if and only if, the holonomy vectors of $S(f)$ are collinear.  
\end{corollary}

\subsection{Trichotomy on the set ${\rm Hol}(S(f))$}\label{Subsubsection:trichotomy}

Since the set of holonomy vectors ${\rm Hol}(S(f))$ of an infinite superelliptic curve $S(f)$ is countable, then it is immediate that its respective closure 
\[
\overline{{\rm Hol}(S(f))},
\]
it is a separable space. Now, as  $\overline{{\rm Hol}(S(f))}$ is a closed subset of the complete metric space $\mathbb{C}$. From a classic result on analysis, we obtain that $\overline{{\rm Hol}(S(f))}$ is also complete. Hence, $\overline{{\rm Hol}(S(f))}$ is a separable completely metrizable topological space. In other words, $\overline{{\rm Hol}(S(f))}$  is a \emph{Polish space}.

It is well-known that the set of holonomy vectors ${\rm Hol}(S)$ of a compact translation surface $S$, it does not have limit points \cite{Voro}*{Proposition 3.1}. In other words, the derived set ${\rm Hol}(S)'$ of ${\rm Hol}(S)$ is empty. Nevertheless, in the case of an infinite superelliptic curve $S(f)$, it could establish the following \textbf{trichotomy} on the set ${\rm Hol}(S(f))$:
\begin{itemize}
    \item[\textbf{(1)}] The set of holonomy vectors ${\rm Hol}(S(f))$ of  $S(f)$ does not have limit points, it means, the derived set ${\rm Hol}(S(f))'=\emptyset$. Therefore, ${\rm Hol}(S(f))$ is discrete.
    \item[\textbf{(2)}] The derived set ${\rm Hol}(S(f))'\neq \emptyset$ and ${\rm Hol}(S(f))$ is a closed subset of $\mathbb{C}$.
    \item[\textbf{(3)}] The derived set ${\rm Hol}(S(f))'\neq \emptyset$ and ${\rm Hol}(S(f))$ is not a closed subset of $\mathbb{C}$.
\end{itemize}
Now, we shall exhibit for each one of the previous cases an infinite superelliptic curve. More precisely, for the case \textbf{(1)}, we shall build an infinite superelliptic curve whose  set of holonomy vectors is a subset of the discrete set $\mathbb{Z}\times\mathbb{Z}$. For the item \textbf{(2)}, we shall show an infinite superelliptic curve $S(f)$ such that its set ${\rm Hol}(S(f))$ is the closed subset 
\[
{\rm Hol}(S(f))=\overline{\left\{\pm\left(1+\frac{1}{k}\right):k\in\mathbb{N}\right\}}\subset \mathbb{C}.
\]
Finally, for the case \textbf{(3)}, we shall introduce an infinite superelliptic curve having to the rational numbers without the zero $\mathbb{Q}\setminus\{\textbf{0}\}$ as  holonomy vectors set and the derived set ${\rm Hol}(S(f))'$ being equal to the irrational numbers with the zero $\mathbb{I}\cup\{\textbf{0}\}$.

\subsubsection{Infinite superelliptic curve $S(f)$ such that its holonomy vectors set ${\rm Hol}(S(f))$ does not have limit points}\label{example:holonomy_Vectors_origami} 

We take the sequences $(z_{k})_{k\in\mathbb{N}}$ such that its points are all complex numbers of the discrete set $\mathbb{Z}\times\mathbb{Z}$. Then, we consider the Weierstrass map $f$ as in the equation \eqref{eq:T_Weierstrass}, which has as simple zeros the points $z_{1},z_{2},\ldots$. Then we consider the infinite superelliptic curve
\[
S(f)=\left\{(z,w)\in\mathbb{C}^2:f(z)=w^n\right\},
\]
with $n\geq 2$. Since the discrete set $\mathbb{Z}\times\mathbb{Z}$ is an additive group and, the appropriate differences of the points of the sequence $(z_{k})_{k\in\mathbb{N}}$ define the holonomy vectors of $S(f)$ (see Theorem \ref{theorem:holonomy_vectors_of_S(f)}), then the set of holonomy vectors ${\rm Hol}(S(f))$ is a discrete subset of $\mathbb{Z}\times\mathbb{Z}$. Thus, the derived set ${\rm Hol}(S(f))'$ of ${\rm Hol}(S(f))$ is empty.   

\subsubsection{Infinite superelliptic curve $S(f)$ such that the derived set ${\rm Hol}(S(f))'\neq \emptyset$ and ${\rm Hol}(S(f))$ is a closed subset of $\mathbb{C}$}\label{example:isc_holonomy_omega}
We consider the sequence $(z_{k})_{k\in\mathbb{N}}$ of complex numbers defined inductively by
\begin{align}\label{eq:seq_hol_vect_compact_ordinal}
    z_{1}&:=\textbf{0},\notag \\
    z_{2}&:=1;\notag\\
    z_{k}&:=z_{k-1}+1+\frac{1}{k-1}, \text{ for each } k\geq 3.
\end{align}
\begin{remark}\label{remark:sequence_colinear_holonomy_vector}
The previously sequence $(z_{k})_{k\in\mathbb{N}}$ satisfies the following properties:
\begin{itemize}
    \item[\textbf{(1)}] Its points are in the real positive line join with the zero.

   \item[\textbf{(2)}] It is a strictly increasing sequence. It means, $z_{1}<\ldots < z_{k}<z_{k+1}<\ldots$. Moreover, $z_{2}-z_{1}=1$ and $z_{k}-z_{k-1}=1+\dfrac{1}{k-1}$ for each $k\geq 3$.

   \item[\textbf{(3)}]  The sequence $(z_{k})_{k\in\mathbb{N}}$ is unbounded, because for each $k\in\mathbb{N}$ it satisfies that $k \leq  z_{n}$ for each $n\geq k+1$.  It implies that the sequence $(z_{k})_{k\in\mathbb{N}}$ is convergent to infinity, it means that $\lim\limits_{k\to\infty} |z_{k}|=\infty$.
\end{itemize}
\end{remark}
 
Let $f$ be the Weierstrass map as in the equation (\ref{eq:T_Weierstrass}) having as simple zeros the points $z_{1},z_{2},\ldots$. Then we consider the infinite superelliptic curve
\[
S(f)=\left\{(z,w)\in\mathbb{C}^2:f(z)=w^n\right\},
\]
with $n\geq 2$. From the Theorem \ref{theorem:holonomy_vectors_of_S(f)} and the properties of the sequence $(z_{k})_{k\in\mathbb{N}}$ described in the Remark \ref{remark:sequence_colinear_holonomy_vector}) it follows that the set of holonomy vectors ${\rm Hol}(S(f))$ of $S(f)$ is the closed (compact) subspace 
\[
{\rm Hol}(S(f))=\left\{\pm 1, \pm\left(1+\frac{1}{k}\right):k\in\mathbb{N}\right\}=\overline{\left\{\pm\left(1+\frac{1}{k}\right):k\in\mathbb{N}\right\}}\subset\mathbb{C},
\]
where the derived set ${\rm Hol}(S(f))'=\{-1,1\}$.

Recall that a countable compact Hausdorff space X has \textbf{characteristic system} $(k; n)$ if its $k$-th
Cantor-Bendixon derivative, that we denote by $X^{k}$, is a finite set of $n$ points, for more details see \cite{kechris}*{p. 33}. For this case, the countable compact set ${\rm Hol}(S(f))$ has characteristic system $(1;2)$. From the theorem in \cite{MaSi}, it follows that ${\rm Hol}(S(f))$ is homeomorphic to the ordinal number $\omega\cdot 2+1$.

\subsubsection{Infinite superelliptic curve $S(f)$ such that the derived set ${\rm Hol}(S(f))'\neq \emptyset$ and ${\rm Hol}(S(f))$ is not a closed subset of $\mathbb{C}$.}\label{example:holonomy_set_rationals}
 As the set of all positive rational numbers join with the zero $\mathbb{Q}^{+}_{\textbf{0}}$ is countably infinite, then there exists a bijective map $\psi:\mathbb{N}\to\mathbb{Q}^{+}_{\textbf{0}}$. We can assume without of generality that $\psi(1)=\textbf{0}$. In addition, $\mathbb{Q}^{+}_{\textbf{0}}$ is thought as a subset of $\mathbb{C}$. Hence, we obtain the divergent sequence $(\psi(n))_{n\in\mathbb{N}}$ of complex numbers. Now, we consider  the infinite series $(s_{k})_{k\in\mathbb{N}}$ generated by $(\psi(n))_{n\in\mathbb{N}}$ where the partial sum $s_{k}$ is given by
 \begin{equation}\label{eq:holonomy_set_rationals}
 s_{k}=\sum\limits_{l=1}^{k}\psi(l).
 \end{equation}
The previously series $(s_{k})_{k\in\mathbb{N}}$ satisfies the following properties:
\begin{itemize}
    \item[\textbf{(1)}] Each one of its terms is in the real positive line join with the zero $\mathbb{R}^{+}_{\textbf{0}}$.

    \item[\textbf{(2)}] The series $(s_{k})_{k\in\mathbb{N}}$ is a strictly increasing sequence, \emph{i.e.}, $s_{k}<s_{k+1}$, for each $k\in\mathbb{N}$. Moreover, it satisfies the relation $s_{k+1}-s_{k}=\psi(k+1)$.

    \item[\textbf{(3)}] The series $(s_{k})_{k\in\mathbb{N}}$ is convergent to infinity. If we fix a real number $M\geq 0$ and natural number $k\in\mathbb{N}$ such that $M\leq k$, as the rational numbers $\mathbb{Q}$ are dense in the real lineal $\mathbb{R}$, then we can find a positive rational number $\psi(N)$, for any $N\in\mathbb{N}$ such that $k<\psi(N)$. From the definition of the series $(s_{k})_{k\in\mathbb{N}}$ and the above property {\rm \textbf{(2)}}, we hold that $k<\psi(N)<s_{N}<s_{n}$, for each $n\geq N$. It shows that the series $(s_{k})_{k\in\mathbb{N}}$ is convergent to infinity \emph{i.e.}, $\lim\limits_{k\to \infty} |s_{k}|=\infty$.
\end{itemize}

Let $f$ be the Weierstrass map as in the equation \eqref{eq:T_Weierstrass} having as simple zeros the points $s_{1},s_{2},\ldots$. Then we consider the infinite superelliptic curve
\[
S(f)=\left\{(z,w)\in\mathbb{C}^2:f(z)=w^n\right\},
\]
with $n\geq 2$. From the Theorem \ref{theorem:holonomy_vectors_of_S(f)} and the properties {\rm \textbf{(1)}} and {\rm \textbf{(2)}} of the series $(s_{k})_{k\in\mathbb{N}}$ described above, it follows that the holonomy vectors set of the infinite superelliptic curve $S(f)$ is the subset 
\[
{\rm Hol}(S(f))=\left\{\pm \psi(k+1):k\in\mathbb{N}\right\}=\mathbb{Q}-\{\textbf{0}\}\subset\mathbb{C}.
\]   
We note that the closure $\overline{{\rm Hol}(S(f))}$ is the real line $\mathbb{R}$.

\section{A characterization for Veech groups of infinite superelliptic curves}\label{section:Veech_group}

\subsection{The Veech group associated to a tame translation surface} Let $S_{1}$ and $S_{2}$ be two tame translation surface. A homeomorphism $T:S_{1}\to S_{2}$ is called \emph{affine diffeomorphism}, if it satisfies the following properties:
\begin{itemize}
    \item[\textbf{(1)}] It sends cone points to cone points.
    \item[\textbf{(2)}] The function $T$ is an real-affine map 
\[
z\mapsto A\cdot z + t, \quad \text{with some } A\in{\rm GL}(2,\mathbb{R}) \text{ and } t\in\mathbb{C},
\]
in the local coordinates of the translation atlas on $S_{1}$ and $S_{2}$, respectively.
\end{itemize}
We denote by ${\rm Aff}_{+}(S)$ the group of all the affine orientation preserving diffeomorphism from the tame translation surface $S$ to itself.

For the case of an infinite hyperelliptic curve $S(f)$, (see equation \eqref{equation-definition:Infinite_hyperelliptic_curve}), we can find some maps in ${\rm Aff}_{+}(S(f))$. Enough, it develops the same ideas about of the theory of covering space described in the proof of the Theorem \ref{theorem:isomorphism} and, as consequence it follows the following result.

\begin{theorem}\label{theorem:some_aff_ori_preser_diff_S(f)}
There exists a fiber-preserving to descend affine orientation preserving diffeomorphism $\tilde{T}: S(f)\to S(f)$ if and only if there exists an affine orientation preserving diffeomorphism $T:\mathbb{C}\to \mathbb{C}$ permuting the zeros of $f$.
\end{theorem}

If we fix a tame translation surface $S$ and choose an element $T\in {\rm Aff}_{+}(S)$, then using the translation structure on $S$, we hold that the differential $dT(p)$ of $T$ at any point $p\in S$ is a constant matrix $A$ belongs to ${\rm GL}_{+}(2,\mathbb{R})$. Then, we define the map
\begin{equation}\label{eq:homomorphism_differential_matrix}
D: {\rm Aff}_{+}(S)\to {\rm GL}_{+}(2,\mathbb{R}),
\end{equation}
where $D(T)$ is the differential matrix of $T$. Using the chain rule, it is easy to verify that $D$ is a group homomorphism. 

\begin{definition}[\cite{Vee}*{p. 556}]\label{def:Veech_group} 
The image of $D$, that we denote by $\Gamma(S)$, is called \textbf{the Veech group of $S$}. 
\end{definition}

Recall that ${\rm GL}_{+}(2,\mathbb{R})$ acts on the set of all the tame translation surfaces by
post composition on charts. More precisely, for every $A \in {\rm GL}_{+}(2,\mathbb{R})$, we obtain the $\mathbb{R}$-linear map $T_{A}:\mathbb{C}\to\mathbb{C}$ given by 
\begin{equation}\label{eq:R-linear_map_A}
z\mapsto A\cdot z.
\end{equation}
For a given translations surfaces $S$, the matrix $A$ defines the translations surfaces $S_{A} := g\cdot S$, from $S$ by compositing each one of its chart with $T_{A}$. So, the identity map 
\[
{\rm Id}: S\to S_{A},
\]
is an affine diffeomorphism with matrix differential $B$, and it induces the group isomorphism $\tilde{{\rm Id}}:{\rm Aff}_{+}(S)\to {\rm Aff}_{+}(S_{A})$ given by
\[
T\mapsto T_{A}\circ T \circ T_{A}^{-1}.
\]
As the differential of $\tilde{{\rm Id}}$ is $A$, then
\[
 \Gamma(S_{A})=A \Gamma(S) A^{-1}.
\]

Using the ideas described above, it is easy to check that, if the tame translation surfaces $S_{1}$ and $S_{2}$ are isomorphic, then their respective Veech group $\Gamma(S_{1})$ and $\Gamma(S_{2})$ also are isomorphic. In particular, if we take $(z_{k})_{k\in\mathbb{N}}$ a convergent sequence to infinity, then the sequence $(T_{A}(z_{k}))_{k\in\mathbb{N}}=(w_{k})_{k\in\mathbb{N}}$ is also convergent to infinity, where $T_{A}$ is as in \eqref{eq:R-linear_map_A}. Hence, we consider the Weierstrass maps $f$ and $f_{A}$ as in the equation \eqref{eq:T_Weierstrass} having as simple zeros the points $z_{1},z_{2},\ldots$ and $w_{1},w_{2},\ldots$, respectively. Then, we obtain the infinite superelliptic curves
\begin{align}
    S(f)&=\left\{(z,w)\in\mathbb{C}^2:w^n=f(z)\right\},\\
    S(f_{A})&=\left\{(z,w)\in\mathbb{C}^2:w^n=f_{A}(z)\right\}.
\end{align}
with $n\geq 2$. Using the ideas described in the proof of the Theorem \ref{theorem:isomorphism}, we can lift the map $T_{A}:\mathbb{C}\to\mathbb{C}$ to an affine diffeomorphism $\tilde{T}_{A}:S(f)\to S(f_{A})$, having matrix differential $A$. Then, the Veech groups of $S(f)$ and $S(f_{A})$ are related as follows: $\Gamma(S(f_{A}))=A \Gamma(S(f)) A^{-1}$.

On the other hand, given that each affine orientation preserving diffeomorphism $T$ of the tame translate surface $S$ sends each saddle connection to a saddle connection, then the Veech group $\Gamma(S)$ associated to $S$ acts on its set of holonomy vectors ${\rm Hol}(S)$ by matrix multiplication,
\[
(A,v)\mapsto Av.
\]

Recall that the Veech group of a compact translate surface is a Fuchsian group (see \cite{Vee}*{Proposition 2.7}, \cite{Voro}*{Proposition 3.3}). In the context of tame translation surfaces there is also a theorem describing all possible subgroups of ${\rm GL}_{+}(2,\mathbb{R})$, which can be Veech groups of a tame translate surface. 

\begin{theorem}[\cite{PSV}*{Theorem 1.1}]\label{T:PSV}
	Let $\Gamma(S)<{\rm GL}_{+}(2,\mathbb{R})$ be the Veech group of a tame translation surface $S$. Then one of the following holds:
	\begin{itemize}
		\item[\textbf{(1)}] The group $\Gamma(S)$ is countable and without contracting elements, it means,  $\Gamma(S)$ is disjoint from the set $\mathcal{U}:=\{A\in {\rm GL}_{+}(2,\mathbb{R}): \vert A \cdot v\vert < \vert v\vert, \text{  for all } v\in\mathbb{C}\setminus \{\textbf{0}\}\}$, or
		\item[\textbf{(2)}] The group $\Gamma(S)$ is conjugated to
		$
		P:=\left\{
		\begin{pmatrix}
		1 & t \\
		0 & s
		\end{pmatrix}\hspace{1mm}:\hspace{1mm} t\in\mathbb{R} \text{  and }
		\hspace{1mm} s\in\mathbb{R}^{+}
		\right\},\hspace{1mm}\text{or}
		$
		\item[\textbf{(3)}] The group $\Gamma(S)$ is conjugated to $P'<{\rm GL_{+}(2, \mathbb{R})}$, the subgroup generated by $P$ and $-{\rm Id}$, or
		\item[\textbf{(4)}] The group $\Gamma(S)$ is equal to ${\rm GL_{+}(2, \mathbb{R})}$.
	\end{itemize} 
\end{theorem}

\subsection{Veech group of infinite superelliptic curves}

In this subsection, we shall give a characterization of Veech groups of infinite superelliptic curves.

Let $(z_{k})_{k\in\mathbb{N}}$ be a convergent sequence to infinite. Let $f$ denote the Weierstrass map as in the equation \eqref{eq:T_Weierstrass} having as simple zeros the points $z_{1},z_{2},\ldots$. Then consider the infinite superelliptic curve 
\begin{equation}\label{eq:superelliptic_curve_Veech_group}
S(f)=\{(z,w)\in\mathbb{C}^2:f(z)=w^n\}.
\end{equation}
with $n\geq 2$. To the complex plane $\mathbb{C}$ and to the infinite superelliptic curve $S(f)$, we remove the points of the sequences $(z_{k})_{k\in\mathbb{N}}$ and the cone points ${\rm Sing}(S(f))$, respectively. Thus, we obtain the subsurfaces 
\begin{align}
\mathbb{C}^{\star} &:=\mathbb{C}\setminus \{z_{k}:k\in\mathbb{N}\},\label{eq1:flute_surface}\\
S(f)^{\star}&:=S(f)\setminus{\rm Sing}(S(f)).\label{eq2:flute_surface_S(f)}
\end{align}
Recall that the translation structures on $\mathbb{C}$ and $S(f)$ respectively, are restricted to the their open subsets $\mathbb{C}^{\star}$ and $S(f)^{\star}$, respectively. Thus,  $\mathbb{C}^{\star}$ and $S(f)^{\star}$ we can thought as  translation surfaces. Moreover, the restriction of the projection map onto the first coordinate 
\begin{equation}\label{eq2:restriction_unramiried_covering_map_pi_z}
p_{\textbf{z}}:S(f)^{\star}\to \mathbb{C}^{\star},
\end{equation}
is an unramified covering map of finite degree $n$ (see Remark \ref{remark:properties_projection_map_first_coordinate}). Recall that the \emph{flute surfaces} are the unique infinite-type surfaces, up to homeomorphism, of genus zero and with space of ends homeomorphic to the ordinal number $\omega + 1$ \cite{Bas93}.

\begin{remark}\label{remark:veech_group_of_the_flute_surface}
We obtain the following properties for the previous translation surfaces.

\begin{itemize}
\item[\textbf{(1)}] The subgroup ${\rm Aff}_{+}(\mathbb{C}^{\star})$ of ${\rm Aff}_{+}(\mathbb{C})$ consisting of all the affine diffeomorphism, permuting the points of the sequence $(z_{k})_{k\in\mathbb{N}}$.

\item[\textbf{(2)}]  The Veech group $\Gamma(\mathbb{C}^{\star})$ of $\mathbb{C}^{\star}$ is the image $D({\rm Aff}_{+}(\mathbb{C}^{\star}))$; where $D$ is the group homeomorphism defined in equation \eqref{eq:homomorphism_differential_matrix}. 
\item[\textbf{(3)}] Given that the tame translation structure on $S(f)$ is defined via the projection map  $p_{\textbf{z}}:S(f)\to \mathbb{C}$, see Theorem \ref{theorem:tame_translate_strucuture_on_S(f)}, and the translation subsurface $S(f)^{\star}$ is obtained from $S(f)$ removing its cone points, then it follows that the groups ${\rm Aff}_{+}(S(f))$ and ${\rm Aff}_{+}(S(f)^{\star})$ are the same. It implies that the Veech groups of $S(f)$ and $S(f)^{\star}$ are also the same. In other words, $\Gamma(S(f))=\Gamma(S(f)^{\star})$.
\end{itemize}
\end{remark}

We shall prove that the Veech groups of the flute surface $\mathbb{C}^{\star}$ and the infinite superelliptic curve $S(f)$ are the same. It means, $\Gamma(S(f))=\Gamma(\mathbb{C}^{\star})$. The following construction is necessary to state and prove our result.

On the other hand, we fix a topological unramified universal covering map of the flute surface $\mathbb{C}^{\star}$,
\begin{equation}\label{eq:unramified_universal_covering}
\tilde{p}:\tilde{S}\to \mathbb{C}^{\star}.
\end{equation}
Then the group of deck transformation ${\rm Deck}(\tilde{S}\,|\,\mathbb{C}^{\star})$ of the covering map $\tilde{p}$ is isomorphic to the fundamental group $\pi_{1}(\mathbb{C}^{\star})$ of the flute  surface $\mathbb{C}^{\star}$ (see \cite{Forster}*{p. 34, 5.6 Theorem}).
In this case, the fundamental group is free on an infinite number of generators (see \cite{ARMORA}*{Theorem 1}), it will be denote by 
\[
\pi_{1}(\mathbb{C}^{\star}):=F_{\infty}.
\]
By Theorem 5.1 in \cite{Massey}*{p. 128}, it follows that there exists an unramified covering map
\begin{equation}\label{eq:unramified_covering_map_q}
q:\tilde{S}\to S(f)^{\star}
\end{equation}
such that we have the following commutative diagram
\begin{equation}\label{diagram:universal_coverin_p}
\begin{tikzcd}
\tilde{S} \arrow{d}[swap]{\tilde{p}} \arrow{r}{q} & S(f)^{\star} \arrow{ld}{p_{\textbf{z}}} \\
\mathbb{C}^{\star}          &
\end{tikzcd}
\end{equation}
it means, $\tilde{p}=p_{\textbf z}\circ q$. 

\begin{remark}\label{remark:group_H}
We note that the group of deck transformation 
\[
H:={\rm Deck}(\tilde{S}\,|\,S(f)^{\star}),
\]
of the covering map $q$ is a subgroup of $F_{\infty}$ such that the quotient space $\tilde{S}/H$ is homeomorphic to $S(f)^{\star}$. Since the unramified covering map $p_{\textbf{z}}$ defined in equation \eqref{eq2:restriction_unramiried_covering_map_pi_z} has degree $n$, then the index of the subgroup $H$ is $n$ \cite{Miranda}*{p. 86}. 
\end{remark}

Using the (topological) unramified universal covering map $\tilde{p}:\tilde{S}\to\mathbb{C}^{\star}$, described in equation \eqref{eq:unramified_universal_covering}, we lift the translation structure of the flute surface $\mathbb{C}^{\star}$ on the universal covering $\tilde{S}$. Recall the notion of developing map associated to the translation surface $\tilde{S}$ (refer to \cite{Thurston1997}*{\S 3.5}, \cite{Goldman}*{\S 5.3.2}). For a fix chart $(U_{i},\varphi_{i})$ of $\tilde{S}$, the \emph{translation developing map} $\textbf{dev}:\tilde{S}\to \mathbb{C}$ is defined as follows:
\[
\varphi_{i}=\textbf{dev}|_{U_{i}}, \quad \text{and} \quad \varphi_{j}=t\circ \textbf{dev}|_{U_{j}} \text{ for a translation} \, t:=t(U_{j},\varphi_{j}),
\]
for any other chart $(U_{j},\varphi_{j})$ of $\tilde{S}$. Given that $\tilde{S}$ is simply connected, then the translation developing map $\textbf{dev}$ is also a covering map. Moreover, if we consider any affine orientation preserving diffeomorphism $\tilde{T}$ of $\tilde{S}$, then there exists a unique affine orientation preserving diffeomorphism $T$ of $\mathbb{C}$ such that the following diagram
\begin{equation}\label{eq:digram_developing_map}
\begin{tikzcd}
\tilde{S}\arrow{d}[swap]{\textbf{dev}} \arrow{r}{\tilde{T}}          &\tilde{S}\arrow{d}{\textbf{dev}}\\
\mathbb{C} \arrow{r}[swap]{T}          &\mathbb{C}
\end{tikzcd}
\end{equation}
 is commutative, \emph{i.e.}, $\textbf{dev}\circ \tilde{T} =T\circ \textbf{dev}$. 
 
\begin{definition} 
We define the group homomorphism
 \begin{equation}\label{eq:map_aff_+}
 {\rm aff}_{+}:{\rm Aff}_{+}(\tilde{S})\to {\rm Aff}_{+}(\mathbb{C}),
 \end{equation}
where ${\rm aff}_{+}$ sends $\tilde{T}\in {\rm Aff}_{+}(\tilde{S})$ to the unique map ${\rm aff}_{+}(\tilde{T}):=T\in {\rm Aff}_{+}(\mathbb{C})$ such that the diagram \eqref{eq:digram_developing_map} is commutative. 
\end{definition}

As each element of the the fundamental group $F_{\infty}$ of the flute surface $\mathbb{C}^{\star}$, is locally the identity, then $F_{\infty}$ is a subgroup of ${\rm Aff}_{+}(\tilde{S})$. Thus, we consider the restriction map
\begin{equation}\label{eq:holonomy_map}
{\rm hol}:={\rm aff}_{+}|_{F_{\infty}}:F_{\infty}\to {\rm Aff}_{+}(\mathbb{C}),
\end{equation}
which is called the \emph{holonomy of} $\tilde{S}$. Now, we consider the group homomorphism
\[
D:{\rm Aff}_{+}(\mathbb{C})\to {\rm GL}_{+}(2,\mathbb{R}),
\]
such that $D(T)$ is the differential matrix of $T\in {\rm Aff}_{+}(\mathbb{C})$. Then so is the composite
\begin{equation}\label{eq:Veech_group_of_universal_covering}
D \circ {\rm aff}_{+}:{\rm Aff}_{+}(\tilde{S})\to {\rm GL}_{+}(2,\mathbb{R}),
\end{equation}
a group homeomorphism, such that its image is the Veech group $\Gamma(\tilde{S})$ of the translation universal covering $\tilde{S}$ (see Definition \ref{def:Veech_group}).

\begin{theorem}\label{theorem:related_Vecch_group_1}
    Let $S(f)$ be an infinite superelliptic curve as in equation \eqref{eq:superelliptic_curve_Veech_group} and let $\tilde{S}$ be the universal covering of the flute surface $\mathbb{C}^{\star}$, endowed with the translation structure defined via the universal covering map $\tilde{p}$ (see equation \eqref{eq:unramified_universal_covering}). Then we have:
    \begin{itemize}
        \item[\textbf{(1)}] $\Gamma(S(f))$ is a subgroup of $\Gamma(U)$.
        \item[\textbf{(2)}] $\Gamma(\mathbb{C}^{\ast})=\Gamma(U)$.
    \end{itemize}
\end{theorem}

\begin{proof}
    \textbf{(1)} We consider the affine orientation preserving diffeomorphism $T\in {\rm Aff}_{+}(S(f))$, then its restriction to $S(f)^{\star}$ is an element in ${\rm Aff}_{+}(S(f)^{\star})$, whose differential matrices are the same (see item \textbf{(4)} in Remark \ref{remark:veech_group_of_the_flute_surface}). We can abuse notation and denote this restriction function by T. Then there exists a lift $\tilde{T}:\tilde{S}\to\tilde{S}$ of $T$ via the unramified covering map $q$, which is described in equation \eqref{eq:unramified_covering_map_q}. Given that translation structure on $U$ is defined from the map $p$, then $\tilde{T}$ is also an affine orientation preserving diffeomorphism of $U$, such that its respective differential matrix coincides with the differential matrix of $\tilde{T}$. In other words, $D(\tilde{T})=D(T)$. It implies that $D(T)\in \Gamma(U)$. Thus, $\Gamma(S(f))$ is a subgroup of $\Gamma(U)$. 

    \textbf{(2)} To prove the equality, the following construction of a developing map is necessary. Recall that identity map ${\rm Id}:\mathbb{C}\to\mathbb{C}$ is an universal covering map, then we take its restriction ${\rm Id}|_{\mathbb{C}^{\star}}$ to $\mathbb{C}^{\star}$. Given that $\tilde{p}=q\circ p_{\textbf{z}}$ is the universal covering map from $\tilde{S}$ to the flute surface $\mathbb{C}^{\star}$, (see equation \eqref{diagram:universal_coverin_p}), then there exist an unramified covering map $h:\tilde{S}\to \mathbb{C}^{\star}$ such that we have the following commutative diagram
\begin{equation}\label{diagram:universal_coverin_h}
\begin{tikzcd}
\tilde{S} \arrow{d}[swap]{\tilde{p}} \arrow{r}{h} & \mathbb{C}^{\star} \arrow{ld}{{\rm Id}|_{\mathbb{C}^{\star}}} \\
\mathbb{C}^{\star}          &
\end{tikzcd}
\end{equation}
it means, ${\rm Id}|_{\mathbb{C}^{\star}}\circ h=\tilde{p}=p_{\textbf z}\circ q$. Since the translation structure on $\tilde{S}$ is obtained by lifting the translation structure on $\mathbb{C}^{\star}$ via the universal covering map $\tilde{p}$, then the unramified covering map $h$ is therefore locally a chart of $\tilde{S}$. From the properties of developing map, it follows that $h$ is also a developing map having as image the open subset $\mathbb{C}^{\star}\subset\mathbb{C}$.

Now, we consider a matrix $A$ in the Veech group $\Gamma(\tilde{S})$, then there exists an element $\tilde{T}\in {\rm Aff}_{+}(\tilde{S})$, such that $A=D\circ {\rm aff}_{+}(\tilde{T})$, (see equation \eqref{eq:Veech_group_of_universal_covering}). By properties of the developing map and the definition of the group homeomorphism ${\rm aff}_{+}$, it follows that  there exists a unique affine orientation preserving diffeomorphism $T$ of $\mathbb{C}^{\ast}$ such that its differential matrix is $A$, and the following diagram
\begin{equation}\label{eq2:digram_developing_map}
\begin{tikzcd}
\tilde{S}\arrow{d}[swap]{h} \arrow{r}{\tilde{T}}          &\tilde{S}\arrow{d}{h}\\
\mathbb{C}^{\star} \arrow{r}[swap]{T}          &\mathbb{C}^{\star}
\end{tikzcd}
\end{equation}
 is commutative, \emph{i.e.}, $h\circ \tilde{T} =T\circ h$. It proves that $A$ is in $\Gamma(\mathbb{C}^{\star})$. Hence, we conclude that $\Gamma(\tilde{S})\subset \Gamma(\mathbb{C}^{\star})$.

 Conversely, we take an element $A\in \Gamma(\mathbb{C}^{\star})$, then there exists an affine diffeomorphism $T\in {\rm Aff}_{+}(\mathbb{C}^{\star})$ having the matrix $A$ as its differential matrix. From the universal covering theory, it follows that the map $T$ can be lifted to some affine orientation preserving diffeomorphism $\tilde{T}$ of $\tilde{S}$ whose differential matrix is $A$. Thus, we obtain that $A\in\Gamma(\mathbb{C}^{\star})$. Hence, we hold the containment $\Gamma(\mathbb{C}^{\star})\subset \Gamma(\tilde{S})$. 
\end{proof}

\begin{theorem}\label{theorem:related_Vecch_group_2}
The Veech group $\Gamma(\mathbb{C}^{\star})$ is a subgroup of $\Gamma(S(f))$.
\end{theorem}

\begin{proof}
It is enough to demonstrate that $\Gamma(\mathbb{C}^{\star})$ is contained in $\Gamma(S(f))$. Let $A$ be a matrix in $\Gamma(\mathbb{C}^{\star})$. From the item \textbf{(1)} in Remark \ref{remark:veech_group_of_the_flute_surface}, it follows that there exists $T$ an affine orientation preserving diffeomorphism of the complex plane $\mathbb{C}$ permuting the zeros $z_{1},z_{2},\ldots$ of the Weierstrass map $f$, (see equation \eqref{eq:superelliptic_curve_Veech_group}), and having differential matrix $D(T)=A$. We shall build a map  $\tilde{T}: S(f)\to S(f)$, which will be defined as a lifting of the map $T$. Then, we will prove that the differential matrix of $\tilde{T}$ is $A$. 

For each $k\in\mathbb{N}$, the image of the cone angle singularity $(z_{k},\textbf{0})$ under $\tilde{T}$ is given by 
 \[
 \tilde{T}(z_{k},\textbf{0}):=(T(z_{k}),\textbf{0}).
 \] 
 Now, we fix a basepoint $(z_{0},w_{0})$ in the subspace $ S(f)^{\star}\subset S(f)$ and select an ordered pair $(u,v)$ from the $n$ ordered pairs conforming the fibre $p_{\textbf{z}}^{-1}(T(z_{0}))$. Then the image of $(z_{0},w_{0})$ under $\tilde{T}$ is defined as
 \[
 \tilde{T}(z_{0}, w_{0}):=(u,v).
 \]
 Given any ordered pair  $(s,t)\in S(f)^{\star}$, then we consider a path $\gamma$ in $\mathbb{C}^{\star}$ having endpoints the complex numbers $z_{0}$ and $s$. Using the unramified covering map $p_{\textbf{z}}:S(f)^{\star}\to \mathbb{C}^{\star}$, the composition map $T\circ \gamma$ can be lifted to a path $\tilde{\gamma}$ of $S(f)^{\star}$ having as endpoints the ordered pairs $\tilde{T}(z_{0},w_{0})$ and $(\tilde{s},\tilde{t})$, such that $(\tilde{s},\tilde{t})$ is in the fiber $p_{\textbf{z}}^{-1}(T(s))$. Thus, the image of the ordered pair $(s,t)$ under $\tilde{T}$ is defined as follows:
 \[
 \tilde{T}(s,t):=(\tilde{s},\tilde{t}).
 \]
 By construction the map $\tilde{T}:S(f)\to S(f)$ is a (homeomorphism) lifting of $T$. Given that the translation structure on $S(f)$ is defined via the projection map $p_{\textbf{z}}$, then locally the map $\tilde{T}$ is a restriction of $T$, it means $T=p_{\tilde{z}}\circ\tilde{T}\circ p_{\textbf{z}}^{-1}$. It implies, that the differential matrix of $\tilde{T}$ is the matrix $A$. Hence, we conclude that $A\in \Gamma(S(f))$.
 \end{proof}

\begin{corollary}\label{corollary:Veech_equality}
 The Veech groups $\Gamma(\mathbb{C}^{\star})$ and $\Gamma(S(f))$ are the same. In other words,  $\Gamma(\mathbb{C}^{\star})=\Gamma(S(f))$.
\end{corollary}

Now, we shall realize the Fuchsian group ${\rm SL}(2,\mathbb{Z})$ as Veech group of an infinite superelliptic curve.

\subsubsection{${\rm SL}(2,\mathbb{Z})$ as Veech group of an infinite superelliptic curve.}
Given that every affine orientation preserving diffeomorphism $T:\mathbb{C}\to \mathbb{C}$ that preserves $\mathbb{Z}\times\mathbb{Z}$, is given by
\[
z\mapsto A\cdot z+t, \quad \text{with some } A\in {\rm SL}(2,\mathbb{Z}) \text{ and } t\in \mathbb{Z}\times\mathbb{Z},
\]
then it follows that the Veech group $\Gamma(S(f))$ of the infinite superelliptic curve $S(f)$ described in \S \ref{example:holonomy_Vectors_origami} is ${\rm SL}(2,\mathbb{Z})$.

\subsection{Uncountable Veech group of infinite hyperelliptic curves}

Let us remark that the only tame translation surfaces having ${\rm GL}_{+}(2,\mathbb{R})$ as their Veech group are the cyclic branched coverings of the complex plane $\mathbb{C}$, (refer to\cite{PSV}*{Lemma 3.2}). Correlating the previous statement and Theorem \ref{T:PSV}, it establishes that the uncountable subgroups of ${\rm GL}_{+}(2,\mathbb{R})$ that can appear as Veech groups of an infinite superelliptic curve are those that are conjugated to $P$ or conjugated to $P'$. We shall introduce a result, which gives sufficient and necessary conditions to guarantee that the Veech group $\Gamma(S(f))$ of an infinite superelliptic curve $S(f)$ is an uncountable proper subgroup $\Gamma(S(f))$ of ${\rm GL}_{+}(2,\mathbb{R})$. Moreover, we shall realize the groups $P$ and $P'$ as Veech group of an infinite superelliptic curves.

We take a sequence of complex numbers $(z_{k})_{k\in\mathbb{N}}$ such that $\lim\limits_{k\to \infty}|z_{k}|=\infty$. Let $f$ denote the Weierstrass map as in equation \eqref{eq:T_Weierstrass} having as simple zeros the points $z_{1},z_{2},\ldots$. Then we consider the infinite superelliptic curve 
\begin{equation}\label{eq:superelliptic_curve_uncountable_Veech_group}
S(f)=\{(z,w)\in\mathbb{C}^2:f(z)=w^n\}.
\end{equation}
with $n\geq 2$. We can suppose without of generality that $z_{1}=\textbf{0}$. We denote by $\Gamma(S(f))$ the Veech group of $S(f)$. The following result gives sufficient and
necessary conditions to guarantee that the Veech group of an infinite superelliptic curve $S (f)$ is an uncountable subgroup of ${\rm GL}_{+}(2, \mathbb{R})$.

\begin{theorem}\label{t:characterization_uncountable}
	The Veech group $\Gamma(S(f))$ of $S(f)$ is an uncountable subgroup of ${\rm GL}_{+}(2,\mathbb{R})$, if and only if, all the points in the sequence $(z_{k})_{k\in\mathbb{N}}$ are collinear. 
\end{theorem}

\begin{proof}
	We shall prove the necessary condition. It is enough to prove that the Veech  $\Gamma(S(f))$ has a conjugated subgroup to $P$. Then we consider the straight line $L$ passing through the origin $\textbf{0}$ such that all the points in the sequence $(z_{k})_{k\in \mathbb{N}}$ lie on that straight line $L$. Now, we consider the rotation complex map $R_{\theta}:\mathbb{C}\to\mathbb{C}$, for any $\theta\in\mathbb{R}$, which sends the straight line $L$ onto the real axis. We remark that for each matrix $A\in P$, the $\mathbb{R}$-affine map $T_{A}:\mathbb{C}\to \mathbb{C}$ defined by 
 \[
 z\mapsto A\cdot z, 
 \]
 fixes the real axis. Then, the composition map $R_{\theta}^{-1}\circ T_{A}\circ R_{\theta}$ is an affine orientation preserving diffeomorphism of $\mathbb{C}$, which fixes the straight line $L$. In particular, each point in the sequence $(z_{k})_{k\in\mathbb{N}}$ is fixed by the map $R_{\theta}^{-1}\circ T_{A}\circ R_{\theta}$. This implies that the position map $R_{\theta}^{-1}\circ T_{A}\circ R_{\theta}$ is an element in ${\rm Aff}_{+}(\mathbb{C}^{\star})$, where $\mathbb{C}^{\star}$ is the flute surface obtained from the complex plane $\mathbb{C}$ by removing all the points in the sequence $(z_{k})_{k\in\mathbb{N}}$. Thus, the differential matrix $D(R_{\theta}^{-1}\circ T_{A}\circ R_{\theta})=D(R_{\theta}^{-1})\cdot A \cdot D(R_{\theta})$ is in the Veech group $\Gamma(\mathbb{C}^{\star})$. Let us recall the equality between Veech groups $\Gamma(\mathbb{C}^{\star})=\Gamma(S(f))$, (see Corollary \ref{corollary:Veech_equality}). Since the matrix $A$ in  $P$ is considered arbitrarily, then the Veech group $\Gamma(S(f))$ contains a conjugated subgroup to $P$ by a rotation matrix. As the group $P$ has uncountable cardinality, then the Veech group $\Gamma(S(f))$ must be uncountable.

 \begin{remark}
  Applying Theorem \ref{theorem:holonomy_vectors_of_S(f)} and Corollary \ref{corollary:vectors_collinear} we obtain the following equivalent statements:
\begin{enumerate}
 \item[\textbf{(1)}] All points in the sequence $(z_{k})_{k\in\mathbb{N}}$ are collinear. 

\item[\textbf{(2)}] The saddles connection of $S(f)$ are parallel.

 \item[\textbf{(3)}] The holonomy vectors of $S(f)$ are parallel. 
\end{enumerate}   
 \end{remark}
 Now, we will show the sufficient condition by proceeding by contradiction on item \textbf{(2)} described in the previous Remark. We assume that the Veech group $\Gamma(S(f))$ of $S(f)$ is an uncountable proper subgroup of ${\rm GL}_{+}(2,\mathbb{R})$, and there are two linearly independent holonomy vectors $v,u \in {\rm Hol}(S(f))$ over $\mathbb{R}$. Given that the vectors $v$ and $u$ form a basis for the real-vector space $\mathbb{\mathbb{C}}$, then the map $\rho$ from $\Gamma(S(f))$ to the product ${\rm Hol}(S(f))\times {\rm Hol}(S(f))$ given by
\[
A \mapsto (A\cdot v, A\cdot u),
\]
it is an injective function. This implies that the product ${\rm Hol}(S(f))\times {\rm Hol}(S(f))$ must be an uncountable set. Clearly, this fact is a contradiction because such product is countable (see Theorem \ref{theorem:holonomy_vectors_of_S(f)}). Hence, we conclude that the holonomy vectors of $S(f)$ are parallel.
\end{proof}

The following corollary is immediate from the Theorem above.

\begin{corollary}
	Let $T_{-{\rm Id}}:\mathbb{C}\to\mathbb{C}$ be the real-affine map defined by the matrix ${\rm -Id}$. If all the points in the sequence $(z_{k})_{k\in\mathbb{N}}$ are invariant under the composition function $R^{-1}_{\theta}\circ T_{-{\rm Id}}\circ R_{\theta}$, then the Veech group $\Gamma(S(f))$ is conjugated to the group $P'$.
\end{corollary}

Now, we shall realize the uncountable groups $P$ and $P'$ as Veech group of infinite superelliptic curves.

\subsubsection{The groups $P$ and $P'$ as Veech group of infinite superelliptic curves.}\label{sub:realization_uncountable_Veech_group}
From the previous results, the Veech group of the following infinite superelliptic curves can be easily computed.

\begin{itemize}
\item[\textbf{(1)}] The group $P$ appears as Veech group of the infinite superelliptic curves described in \S\ref{example:isc_holonomy_omega} and \S\ref{example:holonomy_set_rationals}.

\item[\textbf{(2)}] The group $P'$ can be realized as Veech group of the infinite superelliptic curve $S(f)$, for which the Weierstrass map $f$ as in the equation \eqref{eq:T_Weierstrass} having as simple zeros all the integers numbers $\mathbb{Z}$.
\end{itemize}

\subsection{Countable Veech groups of infinite hyperelliptic curves} 

Given $(z_{k})_{k\in\mathbb{N}}$ an convergent sequence to infinity. Let $f$ be the  Weierstrass maps as in equation \eqref{eq:T_Weierstrass} having as simple zeros the points $z_{1},z_{2},\ldots$. From this holomorphic function $f$, we obtain the infinite superelliptic curve
\begin{equation}\label{eq:inf_super_curv_countable_Veech}
    S(f)=\left\{(z,w)\in\mathbb{C}^{2}:w^{n}=f(z)\right\},
\end{equation}
with $n\geq 2$. The following result provides a precise description of the countable Veech group associated to the infinite superelliptic curve $S(f)$ whose set of holonomy vectors satisfies either \textbf{(1)} or \textbf{(2)} in the \textbf{trichotomy} on the set ${\rm Hol}(S(f))$, see
 \S\ref{Subsubsection:trichotomy}. Recall that $S^{1}:=\{z\in\mathbb{C}:|z|=1\}$  is the unit circle in the complex plane, considered as a topological group under the complex multiplication and the usual topology.

\begin{theorem}\label{theorem:countable_Veech_group_of_infinite_superelliptic_curve}
 Let $\Gamma(S(f))$ be the Veech group of the infinite superelliptic curve $S(f)$ as in equation \eqref{eq:inf_super_curv_countable_Veech}, such that it is an non-trivial countable subgroup  of ${\rm GL}_{+}(2,\mathbb{R})$ without contracting matrices. Then we have:

\begin{itemize}
\item[\textbf{(1)}] If the set of holonomy vectors ${\rm Hol}(S(f))$ does not have limits point, then the Veech group $\Gamma(S(f))$ is discrete.

\item[\textbf{(2)}] If the set of holonomy vectors ${\rm Hol}(S(f))$ is closed such that the derived set ${\rm Hol}(S(f))'\neq \emptyset$, then the Veech group $\Gamma(S(f))$ is homeomorphic to some finite subgroup of $S^{1}$.
\end{itemize}
\end{theorem}

\begin{proof}
\textbf{(1)} We consider the sequences $(A_{k})_{k\in\mathbb{N}}$ of elements in $\Gamma(S)$ convergent to the identity matrix, \emph{i.e.}, 
 $\lim\limits_{k\to\infty}A_{k}={\rm Id}$. Now, we take a holonomy vector $v\in {\rm Hol}(S(f))$. Since the Veech group $\Gamma(S(f))$ acts on ${\rm Hol}(S(f))$ by matrix multiplication, then the sequence  $(A_{k}\cdot v)_{k\in\mathbb{N}}$ of holonomy vectors must be converge to ${\rm Id}\cdot v$. In other words,  $\lim\limits_{K\to \infty} A_{k}\cdot v={\rm Id}\cdot v=v$. Given that the holomony vector set ${\rm Hol}(S(f))$ does not have limits points, and its points are not collinear, then there exists a natural number $M\in\mathbb{N}$ such that for each $m\geq M $ it satisfies $A_{m}={\rm Id}$.    

\textbf{(2)} Given that the countable subspace ${\rm Hol}(S(f))$ is closed and the zero $\textbf{0}\notin {\rm Hol}(S(f))$, then each one of the following sets has minimum
\begin{align}
H_{1}:=\{|z|:z\in {\rm Hol}(S(f))\}\subset\mathbb{R}^{+},\\
H_{2}:=\{|z|:\in {\rm Hol}(S(f))\}\subset \mathbb{R}^{+},
\end{align}
which are different to zero and, they are denoted by $r_{1}$ and $r_{2}$, respectively. Since the norm complex function $G:\mathbb{C}\to \mathbb{R}$, defined by $z\mapsto |z|$, is continuous, then the non-empty sets
\begin{align}
\tilde{H}_{1}:=\{z\in{\rm Hol}(S(f)):|z|=r_{1}\},\\
\tilde{H}_{2}:=\{z\in {\rm Hol}(S(f)): |z|=r_{2}\},
\end{align}
are non-empty closed subsets of ${\rm Hol}(S(f))$. We remark that the points of $\tilde{H}_{i}$ lie on the circle $C_{r_{1}}(\textbf{0})$ with center at $\textbf{0}$ and radius $r_{i}$, with $i\in\{1,2\}$. Recall that each matrix $A$ in the Veech Group $\Gamma(S(f))$ is not contracting, it means $|A\cdot w|>|w|$ or $|A\cdot w|=|w|$, for each $w\in\mathbb{C}\setminus\{\textbf{0}\}$. Given that the Veech group $\Gamma(S(f))$ acts on the holonomy vector set ${\rm Hol}(S(f))$ by matrix multiplication, then $T_{A}$ the action restricted to the matrix $A$ must permute the points of $\tilde{H}_{i}$, for each $i\in\{1,2\}$. This implies that $|A\cdot w|=|w|$ for each $w\in\mathbb{C}\setminus\{\textbf{0}\}$. Therefore, $A$ is an orthogonal matrix having positive determinant. In other words, $A$ is a rotation. Then, we can identify $\Gamma(S(f))$ with a some subgroup of the $S_{1}$. It is well-known that every subgroup of $S^{1}$
 is either dense in $S^{1}$
 or finite. Nevertheless, the Veech group $\Gamma(S(f))$ must be finite, because if $\Gamma(S(f))$ is homeomorphic to some dense subgroup of $S^{1}$, then image of $\tilde{H}_{1}$ under $T_{A}$ the restricted action must be dense in the circle $C_{r_{1}}(\textbf{0})$. Given that the derive set ${\rm Hol}(S(f))'\neq \emptyset$ and ${\rm Hol}(S(f))$ is closed, then we have the equality  $\tilde{H}_{1}=C_{r_{1}}(\textbf{0})$. Since $\tilde{H}_{1}$ is contained in ${\rm Hol}(S(f))'$, it implies that ${\rm Hol}(S(f))$ has the cardinality of the continuum, but ${\rm Hol}(S(f))$ is countable.
\end{proof}

Now, we shall realize a finite subgroup $G< {\rm GL}_{+}(2,\mathbb{R})$, without contracting elements, as Veech group of an infinite superelliptic curve.

\subsubsection{Finite subgroup $G<{\rm GL}_{+}(2,\mathbb{R})$, without contracting elements, as Veech group of an infinite superelliptic curve.}\label{example:finite_group_as_Veech_group}
    
Let $A_{1},\ldots,A_{l}$ be the matrices of the group $G$, for any $l\in\mathbb{N}$ and let $T_{A_{j}}:\mathbb{C}\to\mathbb{C}$ be the $\mathbb{R}$-linear map defined by 
\[
z\mapsto A_{j}\cdot z,
\]
for each $j\in\{1,\ldots,l\}$. We draw on the complex plane $\mathbb{C}$ the straight lines $\ell_{1}$ and $\ell_{i}$, through the point $\textbf{0}$ and parallel to the vectors $1$ and $i$, respectively. Then, we consider the following sets of collinear points 
\begin{align}
C_{1}&=\left\{1\cdot m: m\in\mathbb{N}_{0}\right\}, \\
C_{i}&=\left\{i\cdot m: m\in\mathbb{N}_{0}\right\},
\end{align}
which are belonged to $\ell_{\textbf{1}}$ and $\ell_{\textbf{i}}$, respectively. Now, we define the following sets
\begin{align}
\mathcal{C}&:=\left\{T_{A_{j}}(1\cdot m),\, T_{A_{j}}(i\cdot m): m\in \mathbb{N}_{0} \text{ and } j\in\{1,\ldots, l\}\right\},\\
\mathcal{L}&:=\left\{T_{A_{j}}(\ell_{1}):=\ell_{1}^{j},\, T_{A_{j}}(\ell_{i})=\ell_{i}^{j}: j\in\{1,\ldots, l\}\right\}.
\end{align}
\begin{remark}\label{remark:construction_Veech_finite}
The sets $\mathcal{C}$ and $\mathcal{L}$ have the following properties:
    \begin{itemize}
        \item[\textbf{(1)}] The map $T_{A_{j}}$ permutes all (the straight line, respectively) the points of the set ($\mathcal{L}$, respectively) $\mathcal{C}$, for each $j\in\{1,\ldots,l\}$.
        
        \item[\textbf{(2)}] All the straight lines of $\mathcal{L}$ pass through the point $\textbf{0}$. 
        
         \item[\textbf{(3)}]  As $G$ is finite group without contracting elements, then $\mathcal{C}$ is a countable discrete subset of $\mathbb{C}$. 
   
          \item[\textbf{(4)}] All the points of the set $\mathcal{C}$ define a convergent sequence to infinity $(z_{k})_{k\in\mathbb{N}}$. We can suppose without of generality that $z_{1}=\textbf{0}$. In addition, we obtain the flute surface $\mathbb{C}^{\star}:=\mathbb{C}\setminus\{z_{k}:k\in\mathbb{N}\}$.
    \end{itemize}
\end{remark}
Let $f$ be the Weierstrass map as in equation (\ref{eq:T_Weierstrass}) having  as simple zeros the points $z_{1},z_{2},\ldots$. Thus, we define the infinite superelliptic curve
\[
S(f)=\left\{(z,w)\in\mathbb{C}^2:f(z)=w^n\right\},
\]
with $n\geq 2$.

On the other hand, Corollary \ref{corollary:Veech_equality} guarantees the equality $\Gamma(\mathbb{C}^{\star})=\Gamma(S(f))$. We shall prove that $G=\Gamma(\mathbb{C}^{\star})$. As the $\mathbb{R}$-linear transformation $T_{A_{j}}$ permutes the points of the sequence $(z_{k})_{k\in\mathbb{N}}$, for each $j\in\{1,\ldots,l\}$ (see item \textbf{(1)} in previous Remark \ref{remark:construction_Veech_finite}), then $G$ is a subgroup of $\Gamma(\mathbb{C}^{\star})$. In other words, $G<\Gamma(S(f))$. 

Conversely, we consider an affine orientation preserving diffeomorphism $T:\mathbb{C}\to \mathbb{C}$ permuting the points of the sequence $(z_{k})_{k\in\mathbb{N}}$, such that it is different to the identity. We denote by $A$ the differential matrix associated to $T$, and we will prove that $A\in G$. By construction of the sequence $(z_{k})_{k\in\mathbb{N}}$, we have that there are exactly $2l$ saddle connections of the flute surface $\mathbb{C}^{\star}$,
\[
L_{1}^{1},\, L_{i}^{1},\, \ldots,\, L_{1}^{m},\,L_{i}^{m},
\]
such that one of theirs endpoints is $z_{1}=\textbf{0}$. In addition, the saddles connection $L_{1}^{j}$ and  $L_{i}^{j}$ are contained in the straight lines $\ell_{1}^{j}$ and $\ell_{i}^{j}$, respectively, for each $j\in\{1,\ldots,l\}$. From the Theorem \ref{theorem:holonomy_vectors_of_S(f)}, we obtain that the holonomy vectors of $S(f)$ from these saddle connections are 
\begin{equation}\label{eq:some_holonomy_vectors}
\pm T_{A_{1}}(1),\, \pm T_{A_{1}}(i),\, \ldots, \, \pm T_{A_{l}}(1),\, \pm T_{A_{l}}(i).
\end{equation}

Now, given that the Veech group $\Gamma(S(f))=\Gamma(\mathbb{C}^{\star})$ acts on the set of holonomy vectors ${\rm Hol}(S(f))$ of $S(f)$, then the matrix $A$ must leave invariant the vectors described in equation \eqref{eq:some_holonomy_vectors}. Moreover, the $\mathbb{R}$-linear transformation $T_{A}:\mathbb{C}\to\mathbb{C}$ given by 
\[
z\mapsto A\cdot z
\] 
must permute the straight lines of the set $\mathbb{L}$. The only possibility is $A=A_{j}$, for any $j\in\{1,\ldots,l\}$. Therefore, $G<\Gamma(S(f))$.

\begin{question}[Open]
  Which are the subgroups of ${\rm GL}_{+}(2,\mathbb{R})$ appearing as Veech groups of an infinite superelliptic curve satisfying \textbf{(3)} in \S\ref{Subsubsection:trichotomy} the \textbf{trichotomy} on the set ${\rm Hol}(S(f))$?  
\end{question}

\subsection*{Acknowledgements}

We express our gratitude to UNIVERSIDAD NACIONAL DE COLOMBIA, SEDE MANIZALES. He has dedicated this work to his beautiful family: Marbella and Emilio, in appreciation of their love and support. We thank Israel Morales Jim\'enez and Alexander Arredondo for providing helpful comments and conversations
on this work.


\end{document}